\documentclass[12pt]{article}
\usepackage{cite,amsfonts,amsmath,amsbsy,amssymb,amsthm,setspace} %,mathtools
\newtheorem{proposition}{Proposition}
\newtheorem{theorem}{Theorem}
\newtheorem{definition}{Definition}

\newtheorem{corollary}{Corollary}
\newcommand {\N}{\mathbb{N}}
\newcommand {\T}{\mathbb{T}}

\newcommand{\p}{\psi}
\newcommand{\f}{\varphi}
\newcommand{\F}{\Phi}

\begin{document}        % What came before is the 'preamble'
\title{Mixing endomorphisms on toroidal groups and their countable products}
\author{John Burke, Leonardo Pinheiro}%\textsuperscript{a}\thanks{$^\ast$Corresponding
%author. Email: lpinheiro@ric.edu}}
%\affil{\textsuperscript{a}Rhode Island College, Providence, RI, USA

%\received{v5.0 released February 2015} }
\maketitle

%-----------------------------------------------------------------------------
%-----------------------------------------------------------------------------
\begin{abstract}
%We apply results by Chan \cite{CHAN01} and Moothatu \cite{MOO09} inspired by the study of dynamics of linear operators to the setting of continuous endomorphisms on topological groups.

We show that all non-trivial continuous endomorphisms of the circle group are topologically mixing. We also show that there exists a large infinite class of continuous endomorphisms of any n-dimensional torus group which are topologically mixing.

Lastly, we prove that any continuous endomorphism on an abelian polish semigroup (with an identity) can be extended in a natural way to a topologically mixing endomorphism on the countable infinite product of said semigroup. This shows that every countable infinite product of an abelian polish semigroup has a topologically mixing endomorphism and, in particular, the countable infinite toroidal group has infinitely many topologically mixing endomorphisms.
\end{abstract}

\section{Introduction}

The theory of discrete dynamical systems is concerned with the the behavior of the iterates of a continuous map on a (usually compact) metric space. The most interesting and studied examples include maps that, in some sense, `mix' the space. For a nice survey on the subject see the article by Kolyada and Snoha \cite{Kolyada97}.

Formally, let $X$ be a topological space, and let $f$ be a continuous map from $X$ to $X$, write
$$f^n(x)=\underbrace{f\circ f \circ \cdots \circ f}_{n-fold}$$
to denote the $n^{th}$ iteration of the map $f$.

We say $f$ is topologically transitive if given any two non-empty open sets $U$ and $V$ of $X$ there exists a natural number $n$ such that $f^n(U) \cap V \neq \emptyset.$

A continuous map $f:X\to X$ is said to be topologically mixing (or just mixing) if given any two non-empty open sets $U$ and $V$ of $X$ there exists a natural number $N$ such that $f^n(U) \cap V \neq \emptyset$, whenever $n>N$. Not surprisingly, mixing is a stronger condition than topological transitivity. The irrational rotation of the circle is a topologically transitive map that is not mixing.  

%It is also a well known fact \cite{Bir22} that for a continuous map $f: X \to X$ with $X$ void of isolated points, $f$ is topologically transitive if and only if there exists a point $x \in X$ for which its orbit under $f$ $orb (x,f)=\{x,f(x), f^2(x), f^ 3(x), ...\}$ is dense in $X$.

It is important to note that the general theory of discrete dynamical systems is usually not concerned with any underlying algebraic structure of the space $X$. Operator theorists, on the other hand, are usually interested in the dynamics of maps preserving the linear structure of the underlying space. In this vein, we will study the dynamics of continuous endomorphisms on topological groups.

In the setting of linear operators acting on a Fr\'echet space, a very celebrated result is the set of sufficient conditions for an operator to be mixing known as the Hypercyclicity Criterion. The result first appeared in Kitai \cite{KITAI82} and was later independently rediscovered by Gethner and Shapiro \cite{GETHNER87}. Chan \cite{CHAN01} and Moothatu \cite{MOO09} independently proved that there is an analogous set of sufficient conditions for a continuous epimorphism to be mixing in the setting of topological groups. These conditions will be discussed in the following section.

\section{Main Results}

In what follows $G$ will denote a  metric, complete, separable topological semigroup with identity: a polish semigroup for short.

Recall that for a semigroup $G$ with the semigroup operation written as multiplication, an endomorphism is a map $\f:G\to G$ such that $\f(gh)=\f(g)\f(h)$ for all $g$ and $h$ in $G$. When $G$ in endowed with a topology compatible with the semigroup structure, we can study the behavior under iteration of such maps.

We introduce the following definition:

\begin {definition}[Semigroup Mixing Criterion]
Let $G$ be a polish semigroup with identity $e$.  We say that a continuous endomorphism $\f:G \to G$ satisfies the Mixing Criterion if there exists dense sets $F$ and $H$ of $G$, and maps $\p_n:F\to F$ such that, for any $f \in F$ and $h \in H$:

(i) $\f^{n}(h)\to e$ as $n\to \infty$.

(ii) $\p_n (f) \to e$ as $n \to \infty$.

(iii) $\f^{n}\p_n (f)\to f$ as $n \to \infty$.

\end {definition}

Chan \cite{CHAN01} and Moothatu \cite{MOO09} independently showed that if $\f$ is a continuous endomorphism on a polish semigroup satisfying the Mixing Criterion, then $\f$ is topologically mixing. Formally, we have.

\begin{theorem}[Chan \cite{CHAN01} and Moothatu \cite{MOO09}]
Let  $G$ be a polish semigroup with identity $e$ and $f:G \to G$ a continuous endomorphism.  If $\f$ satisfies the Semigroup Mixing Criterion then $\f$ is mixing.
\end{theorem}

We now apply Theorem 1 to a particular family of maps on the complex unit circle. We will denote by $\T$ the set of complex numbers of modulus one endowed with the topology induce by the arclength metric, i.e, the distance between two points in $\T$ is given by the length of the shortest arc joining them. Complex multiplication is then compatible with this topology and we have that $\T$ is a compact polish group.  We will show that the homomorphisms of the form $f(z)=z^n$, where $n \geq 2$, are weakly mixing.  Indeed, we have:

\begin{proposition}
Fix $n \in \N$ such that $n \geq 2$ and let $f:\T \to \T$ be given by $f(z)=z^n$, then f is mixing.
\end{proposition}
\begin{proof}
Notice that $f(z)=z^n$ is clearly a continuous endomorphim on $\T$; we will show that it is mixing by showing it satisfies the hypothesis of the Semigroup Mixing Criterion.  First recall that every element $ z\in \T$ can be written uniquely as $z=e^{i\theta}$ with $0\leq\theta<2\pi$.  Now, consider the set
\[
F_k^n=\left\{z \in \T ; z^{n^k}=1\right\}
\]
We claim that $F^n=\cup_{i=1}^{\infty}F_i^n$ is dense in $\T$.

Note that $F_i^n$ is a set of $n^i$ evenly spaced points on the unit circle. Thus, $F_i^n$ partitions the unit circle into $n^i$ arcs of the same length. Also note that $F_{i+1}^n$ contains the points in $F_i$ and further subdivides each arc into $n$ arcs of the same length. Consider some fixed $z \in \T$ and an arbitrary $\epsilon>0$. We can find $i$ large enough so that the smallest distance between two points in $F_i^n$ is less than $\epsilon$ and thus the minimal distance between any element of $\T$ and some element of $F_i^n$ is less than $\epsilon$. Thus at least one element of $F_i \subseteq F$ must be in the $\epsilon$-ball around $z$.

Now, if we iterate the map $f$, we get that

\begin{align*}
& f(z)=z^n, \\
& f^2(z)=f(f(z))=(z^n)^n=z^{n^2}, \\
& f^3(z)=f(f^2(z))=(z^{n^2})^n=z^{n^3}\\
& \vdots\\
& f^r(z)=z^{n^r}.\\
\end{align*}
If $z \in F^n$ (thus $z \in F_k^n$ for some $k$), then for large $N$, $f^N(z)=f^{N-n^k}(f^{n^k}(z))=f^{N-n^k}(1)=1$ so $ f^{r}(z) \to 1$ as $r \rightarrow \infty$.

Now, for each $r \in \N$, define $\p_r: \T \to \T$ by $\p_r(e^{i\theta})=e^{i{\frac{\theta}{n^r}}}$. Observe that $\p_r(z) \to 1$ as $r \to \infty$ for all $ z \in \T$ and that $f^{r}(\p_{r}(z))=z$ for all $z \in \T$.  Hence, $f$ is mixing.
\end{proof}

It is a well known result \cite{RepTheory} that $f(z) = z^n$ where $n \geq 2$ are the only continuous endomorphisms of $\T$. Thus, we have the following corollary.

\begin{corollary}
Every continuous epimorphism of $\T$, except for the identity, is mixing.
\end{corollary}

We now consider the continuous epimorphisms of $\prod_{i=1}^{k} \T \equiv \T^k$. Since Hom(G,H) distributes over direct sums, we know that all continuous epimorphisms  $f: \T^n \rightarrow \T^n$ are of the form $$f(z_1, z_2, \cdots, z_n)= (z_1^{m_{1,1}} z_2^{m_{1,2}} \cdots z_n^{m_{1,n}}, z_1^{m_{2,1}} z_2^{m_{2,2}} \cdots z_n^{m_{2,n}}, \cdots, z_1^{m_{n,1}} z_2^{m_{n,2}} \cdots z_n^{m_{n,n}})$$ where $m_{i,j}$ is a nonnegative integer.  We can show that a large class of such maps is mixing. We have:

%, its complement is open. Thus, a function of this form cannot be mixing since the n$^{th}$ iteration, $n > 0$, of $f(U)$, for some open set $U$ will not intersect the complement of the diagonal.

\begin{proposition}
Let $\sigma$ be a permutation on the set $\{1, 2, \cdots, n\}$. Then the map $f: \T^k \rightarrow \T^k$ defined by $f(z_1, z_2, \cdots, z_k) = (z_{\sigma^{-1}(1)}^{m_1}, z_{\sigma^{-1}(2)}^{m_2}, \cdots, z_{\sigma^{-1}(k)}^{m_k})$  where $m_i\in \N$ and $m_i \geq 2$, for all $i$, is mixing if $\gcd\{m_j | j \in \widetilde{\sigma}(i)\} > 1$ for all $i$ (where $\widetilde{\sigma}(i)$ is the orbit of i) or if $\sigma$ is the identity permutation.
\end{proposition}

\begin{proof}
%{\bf Case $i$)}
If $f(z_1, z_2. \cdots, x_k)=(z_1^ {m_1}, z_2^ {m_2}, \cdots, z_k^ {m_k})$ then $f$ is the product of topologically mixing maps and hence, topologically mixing. \cite{LinearChaos}.

Now, suppose $\gcd\{m_j | j \in \widetilde{\sigma}(i)\} > 1$ for all $i$.

Notice that $f$ is a continuous endomorphism since it can be viewed as a product of continuous endomorphism composed with a permutation of coordinates. We will show $f$ is mixing by showing it satisfies the Semigroup Mixing Criterion.

Let $F^s = \cup_{i=1}^{\infty}F_r^s$ where $F_r^s$ is defined as it was in the proof of Proposition 1. Let $s_i = \gcd\{m_j | j \in \widetilde{\sigma}(i)\}$ and let $F = F^{s_1} \times F^{s_2} \times \cdots \times F^{s_k}$. Note that $F$ is a dense subset of $\T^n$.

We will now show that for $z=(z_1, z_2, ..., z_k) \in F$, $f^n(z_1, z_2, \cdots, z_k) \rightarrow 1$. Observe that
\[
\begin{aligned}
 f^n(z_1, z_2, \cdots, z_k) &= \big(\big(\big(\big(z_{\sigma^{-n}(1)}^{m_{\sigma^{-(n-1)}(1)}}\big)^{\cdots}\big)^{m_{\sigma^{-1}(1)}}\big)^{m_1}, \cdots,  \big(\big(\big(z_{\sigma^{-n}(k)}^{m_{\sigma^{-(n-1)}(k)}}\big)^{\cdots}\big)^{m_{\sigma^{-1}(k)}}\big)^{m_k}\big)\\
&= \big(z_{\sigma^{-n}(1)}^{{m_{\sigma^{-(n-1)}(1)}}\cdots{m_{\sigma^{-1}(1)}}{m_1}}, \cdots, z_{\sigma^{-n}(k)}^{{m_{\sigma^{-(n-1)}(k)}}\cdots{m_{\sigma^{-1}(k)}}{m_k}}\big).
\end{aligned}
\]
Note that the $i^{th}$ coordinate of the image of

$f^n(z_1, z_2, \cdots, z_k)$ is $z_{\sigma^{-n}(i)}^{{m_{\sigma^{-(n-1)}(i)}}\cdots{m_{\sigma^{-1}(i)}}{m_i}}$. For instance if we let $k=5$ and let $\sigma = (1,3,5,2,4)$. Then $f(z_1, z_2, \cdots, z_k) = (z_4^{m_1}, z_5^{m_2}, z_1^{m_3}, z_2^{m_4}, z_3^{m_5})$ and the $2^{nd}$ coordinate of $f^3(z_1, z_2, \cdots, z_k)$ is $\big(\big(z_{\sigma^{-3}(2)}^{m_{\sigma^{-2}(2)}}\big)^{m_{\sigma^{-1}(2)}}\big)^{m_2} = {{{z_1}^{m_3}}^{m_5}}^{m_2}$

Note that $s_{\sigma^{-(n)}(i)}^n \bigm| ({{m_{\sigma^{-(n-1)}(i)}}\cdots{m_{\sigma^{-1}(i)}}{m_i}})$. Since $z_{\sigma^{-N}(i)} \in F^{s_{\sigma^{-N}(i)}}$,  for large $N$ the $i^{th}$ coordinate of $f^N$ is
$$z_{\sigma^{-N}(i)}^{{{m_{\sigma^{-(N-1)}(i)}}\cdots{m_{\sigma^{-1}(i)}}{m_i}}} = \big(z_{\sigma^{-N}(i)}^{s_{\sigma^{-(N)}(i)}^N}\big)^w = 1^w = 1$$ for all $i$, where $w$ is some positive integer. Thus $f^n \rightarrow 1$ as $n \rightarrow \infty$.

Now for each $n \in \N$ define $\p_n : \T^k \rightarrow \T^k$ by $$\p_n(z_1, \cdots, z_k) = z_{\sigma^{n}(1)}^{(-m_{\sigma^n(1)}){\cdots}(-m_{\sigma(1)})}, \cdots, z_{\sigma^{n}(k)}^{(-m_{\sigma^n(k)}){\cdots}(-m_{\sigma(k)})}.$$

For instance if we let $k=5$ and let $\sigma = (1,3,5,2,4)$. Then the $1^{st}$ coordinate of $\p_3(z_1, z_2, \cdots, z_k)$ is $z_{\sigma^{3}(1)}^{(-m_{\sigma^3(1)})(-m_{\sigma^2(1)})(-m_{\sigma(1)})} = {{{z_2}^{-m_2}}^{-m_5}}^{-m_3}$
Observe that since $m_i>2$ for all $i$, $\p_n(z_1, \cdots, z_k) \rightarrow 1$ as $n \rightarrow \infty$ and that $f^n(\p_n(z)) = z$ for all $z \in \T^k$. Hence $f$ is mixing.

\end{proof}

Notice that not all continuous endomorphisms of this form are mixing. For instance, when $m_{i,j} = k_j$ for all $i$ , we get $$f(z_1, z_2, \cdots, z_n) =
(z_1^{k_1} z_2^{k_2} \cdots z_n^{k_n}, z_1^{k_1} z_2^{k_2} \cdots z_n^{k_n}, \cdots, z_1^{k_1} z_2^{k_2} \cdots z_n^{k_n})$$
The  range of $f$ is a subset of the diagonal of $\T^k$ (elements of the form $(z,z, \cdots, z)$). Since the diagonal of $\T^k$ is a closed subset, its complement is and open set which does not intersect the orbit of any element in the diagonal. Hence $T$ cannot be mixing.
%\begin{proposition}
%MAYBE!!!  $f(z_1, z_2, \cdots, z_n) =$ $(z_1^k z_2^k \cdots z_n^k, z_1^k z_2^k \cdots z_n^k, \cdots, z_1^k z_2^k \cdots z_n^k)$ is degenerate mixing on the diagonal
%\end{proposition}
%
%IDEA: After one iteration is of the form $(z_1^{m_1}, z_2^{m_2}, \cdots, z_n^{m_n})$

A natural question is whether any polish group supports a mixing endomorphism.  We show that if the group in question can be written as a countable infinite product of isomorphic copies of one of its closed subgroups, then the answer is affirmative.  More precisely, we have the following theorem:% inspired by the work of Chan and Turcu \cite{CHAN11} related to chaotic extensions:

\begin{theorem}\label{InfiniteThm}
\label{ext}
Let $G$ be a metrizable separable topological group and let $\f:G \to G$ be a continuous endomorphism. There exists a continuous endomorphism $\F: \prod_{i=0}^{\infty} G \to \prod_{i=0}^{\infty} G$ such that
\item[(i)] $\F$ is mixing, and
\item[(ii)]$\pi_0 \circ \F=\pi \circ \f$ where $\pi_0: \prod_{i=0}^{\infty} G \to G$ is the natural projection to the $0^{th}$ coordinate.
\end{theorem}

\begin{proof}
We will construct the endomorphism $\F$ and check it is mixing by verifying it satisfies the Semigroup Mixing Criterion.

Notice if we consider a metric $d$ on $G$ which is bounded by 1, we  can define a metric $\rho$ on $\prod_{i=0}^{\infty} G$ via
 \[
 \rho(g,h)=\sum_{i=0}^{\infty}\frac{d(g_i,h_i)}{2^i}
 \]

For each $g\in \prod_{i=0}^{\infty} G$ write
\[
g=(g_0,g_1, g_2,\cdots),\text{   } g_i \in G,
\]
and define the maps
\[
\Phi(g)=\left(\f(g_0)g_1, g_2, g_3, \cdots \right)
\] and

\[
\Psi(g)=\left(e,g_0,g_1,g_2, \cdots \right)
\]
where $e$ is the identity in $G$.

It is clear that $\F$ is a continuous endomorphism on $\prod_{i=0}^{\infty} G$ with $\pi_0 \circ \F=\pi_0 \circ \f$. Also, for all $g\in \prod_{i=0}^{\infty} G$, $\F(\Psi(g))=g$ and $\Psi^ng=\tilde{e}$  as $n\to \infty$ where the identity element $\tilde{e}$ of $\prod_{i=0}^{\infty} G$ is  $(e,e,e,\cdots)$. We will now verify that $\F$ is mixing. Let $H$ be a dense set in $G$ and consider the subgroup $\tilde{D}$ in $\prod_{i=0}^{\infty} G$ whose elements are of the form $\left(h_0, h_1, h_2, \cdots,h_k, e, e, e,\dots \right)$ for some natural $k$ and $h_i \in H$.  This is clearly dense in $\prod_{i=0}^{\infty} G$ and for any element $h=\left(h_0, h_1, h_2, \cdots,h_k, e, e, e,\dots \right) \in \tilde{D}$, we have
\[
\F^k(h)=\left(\phi^k(h_0)\f^{k-1}(h_1)\cdots\f(h_{k-1})h_k, e, e, \cdots \right)
\]and

\[
\Psi^k(\Phi^k(h))=(\underbrace{e,e, \cdots, e}_\text{ k  positions},\phi^k(h_0)\phi^{k-1}(h_1)\cdots\phi(h_{k-1})d_k, e, e, \cdots).
\]
and

Notice that
 \[
 \begin{aligned}
 \rho(\Psi^k(\Phi^k(h)), \tilde{e}) &=\frac{d(\phi^k(h_0)\phi^{k-1}(h_1)\cdots\phi(h_{k-1})h_k,e)}{2^k} \\
 &\leq\frac{1}{2^k} \to 0 \text{ as } k \to \infty.
 \end{aligned}.
 \]

 Now, consider the set
 \[ \tilde{C}=\{g(\Psi^n(\Phi^n(g))^{-1},g \in \tilde{D}, n \in \mathbb{N}\}.
 \] Notice that
 \[
 \begin{aligned}
 \lim_{n\to \infty}g(\Psi^n(\Phi^n(g))^{-1}&=g\lim_{n\to \infty}(\Psi^n(\Phi^n(g))^{-1}\\
  &=g(\lim_{n\to \infty}(\Psi^n(\Phi^n(g))^{-1}\\
 &=ge \\
 &=g
 \end{aligned}
 \] which shows that $\tilde{C}$ is dense is $\tilde{G}$.

Also, let $c \in \tilde{C}$, so $c=g(\Psi^n(\Phi^n(g))^{-1}$ for some $g \in \tilde{G}$ and we have
\[
\begin{aligned}
\Phi^n(c) &=\Phi^n(g(\Psi^n(\Phi^n(g))^{-1})\\
&=\Phi^n(g)(\Phi^n(\Psi^n(\Phi^n(g))))^{-1}\\
%&=\Phi^n(g)(\Phi^n(g)\Psi^n(g))^{-1}(\Phi^n(g))^{-1}\\
&=\Phi^n(g)\tilde{e}(\Phi^n(g))^{-1}\\
&=\tilde{e}.
\end{aligned}
\] Hence, $\Phi^n \to \tilde{e}$ on $\tilde{C}$ and by Theorem, we conclude $\Phi$ is mixing.
 \end{proof}

Theorem \ref{InfiniteThm} in particular shows that every continuous endomorphism of $\T$ (most of which are mixing) has an extension to a mixing endomorphism of the countable infinite toroidal group, $\T^{\infty}$. In fact, every continuous endomorphism of $\T^k$ (some of which are not mixing) has an extension to a mixing endomorphism of $\T^{\infty}$ since the $\T^{\infty}$ is the countable infinite product of any $\T^k$. Lastly, the theorem shows that every countable infinite product of a polish semigroup admits a topologically mixing endomorphism, since the identity map or the trivial map on the semigroup can be extended to mixing endomorphism on the product.

\section{Future Work}

A natural extension of our result would be to characterize all the mixing maps on the finite and infinite toroidal groups.

\bibliographystyle{plain}
\bibliography{references}

\begin{thebibliography}{1}

\bibitem{LinearChaos}
A.~Peris K.~Grosse-Erdmann.
\newblock {\em Linear Chaos}, chapter~1, pages 17--18.
\newblock Springer, 2011.

\bibitem{CHAN01}
K.Chan.
\newblock Universal meromorphic functions.
\newblock {\em Complex Variables, Theory and Applications}, 46:307--314, 2001.

\bibitem{KITAI82}
C.~Kitai.
\newblock Invariant closed sets for linear operators.
\newblock {\em University of Toronto}, Thesis, 1982.

\bibitem{RepTheory}
Emmanuel Kowalski.
\newblock {\em An Introduction to the Representation Theory of Groups}.

\bibitem{GETHNER87}
J.~Shapiro M.Gethner.
\newblock Universal vectors for operators on spaces of holomorphic functions.
\newblock {\em Proceedings of the American Mathematical Society}, 100:281--288,
  1987.

\bibitem{MOO09}
T.K.Subrahmonian Moothathu.
\newblock Weak mixing and mixing of a single transformation of a topological
  (semi)group.
\newblock {\em Aequationes mathematicae}, 78:147--155, 2009.

\bibitem{Kolyada97}
L.~Snoha S.~Kolyada.
\newblock Topological transitivity - a survey.
\newblock {\em Grazer Math. Ber.}, 334:3--35, 1997.

\end{thebibliography}

%\bibliography{references.bib}{}
%\bibliographystyle{plain}
\end{document}